\documentclass[pdflatex,sn-mathphys]{sn-jnl}
\pdfoutput=1
\usepackage[capitalize]{cleveref}
\usepackage{stmaryrd}  
\usepackage{mathtools} 
\newcommand{\iprod}[1]{\langle#1\rangle}

\newcommand{\jump}[1]{\llbracket#1\rrbracket}
\newcommand{\ud}{\mathrm{d}}
\newcommand{\bs}[1]{\boldsymbol{#1}}
\newcommand{\fine}{\mathrm{fine}}

\newcommand{\UR}{U^\mathrm{R}}
\newcommand{\bUR}{\bs{U}^\mathrm{R}}
\theoremstyle{thmstyleone}
\newtheorem{theorem}{Theorem}[section]
\newtheorem{lemma}[theorem]{Lemma}
\newtheorem{corollary}[theorem]{Corollary}
\theoremstyle{thmstyletwo}
\newtheorem{example}[theorem]{Example}
\newtheorem{remark}[theorem]{Remark}
\begin{document}
\title[Discontinuous Galerkin Time Stepping]%
{Error Profile for Discontinuous Galerkin Time Stepping of Parabolic PDEs}
\author*[1]{\fnm{William} \sur{McLean}}\email{w.mclean@unsw.edu.au}
\author[1]{\fnm{Kassem} \sur{Mustapha}}\email{kassem.ahmad.mustapha@gmail.com}
\affil*[1]{\orgdiv{School of Mathematics and Statistics},
\orgname{University of New South Wales},
\orgaddress{\city{Kensington}, \postcode{2052}, \state{NSW},
\country{Australia}}}

\abstract{We consider the time discretization of a linear parabolic problem by
the discontinuous Galerkin (DG) method using piecewise polynomials of degree at
most~$r-1$ in~$t$, for $r\ge1$ and with maximum step size~$k$. It is well known
that the spatial $L_2$-norm of the DG error is of optimal order~$k^r$ globally
in time, and is, for~$r\ge2$, superconvergent of order~$k^{2r-1}$ at the nodes.
We show that on the $n$th subinterval~$(t_{n-1},t_n)$, the dominant term in the
DG error is proportional to the local right Radau polynomial of degree~$r$.
This error profile implies that the DG error is of order~$k^{r+1}$ at the
right-hand Gauss--Radau quadrature points in each interval.  We show that the
norm of the jump in the DG solution at the left end point~$t_{n-1}$ provides an
accurate \emph{a posteriori} estimate for the maximum error over the
subinterval~$(t_{n-1},t_n)$.  Furthermore, a simple post-processing step yields
a \emph{continuous} piecewise polynomial of degree~$r$ with the optimal global
convergence rate of order~$k^{r+1}$. We illustrate these results with some
numerical experiments.}

\keywords{Superconvergence, Post-processing, Gauss--Radau quadrature, Legendre
polynomials}
\pacs[MSC Classification]{%
65J08, 
65M15} 
\maketitle
\section{Introduction}

Consider an abstract, linear initial-value problem
\begin{equation}\label{eq: IVP}
u'(t)+Au(t)=f(t)\quad\text{for $0<t\le T$,}\quad\text{with $u(0)=u_0$.}
\end{equation}
We assume a continuous solution~$u:[0,T]\to\mathbb{L}$, with
$u(t)\in\mathbb{H}$ if~$t>0$, for two Hilbert spaces $\mathbb{L}$~and
$\mathbb{H}$ with a compact and dense imbedding $\mathbb{H}\subseteq\mathbb{L}$.
By using the inner product~$\iprod{\cdot,\cdot}$ in~$\mathbb{L}$ to identify
this space with its dual~$\mathbb{L}^*$, we obtain an
imbedding $\mathbb{L}\subseteq\mathbb{H}^*$.  The linear
operator~$A:\mathbb{H}\to\mathbb{H}^*$ is assumed to be bounded and
self-adjoint, as well as strictly positive-definite.  For instance, if
$A=-\nabla^2$ so that \eqref{eq: IVP} is the classical heat equation on a
bounded Lipschitz domain~$\Omega\subset\mathbb{R}^d$ where~$d\ge1$, and if
we impose homogeneous Dirichlet boundary conditions, then in the usual way we
can choose $\mathbb{L}=L_2(\Omega)$~and $\mathbb{H}=H^1_0(\Omega)$, in which
case $\mathbb{H}^*=H^{-1}(\Omega)$.

For an integer~$r\ge1$, let $U$ denote the discontinuous Galerkin (DG)
time-stepping solution to~\eqref{eq: IVP} using piecewise-polynomials of degree
at most~$r-1$ with coefficients in~$\mathbb{H}$.  Thus, we consider only the
time discretization with no additional error arising from a spatial
discretization. \cref{sec: DG} summarizes known results on the convergence
properties of the DG solution~$U$, and \cref{sec: Legendre} introduces a local
Legendre polynomial basis that is convenient for the practical implementation
of DG time stepping as well as for our theoretical study.  These sections serve
as preparation for \cref{sec: Error behaviour} where we show that
\begin{equation}\label{eq: error}
U(t)-u(t)=-a_{nr}(u)\bigl[p_{nr}(t)-p_{n,r-1}(t)\bigr]+O(k_n^{r+1})
    \quad\text{for $t\in I_n$.}
\end{equation}
Here, $k_n$ denotes the length of the $n$th time interval~$I_n=(t_{n-1},t_n)$,
the function~$p_{nr}$ denotes the Legendre polynomial of degree~$r$, shifted
to~$I_n$, and $a_{nr}(u)$ denotes the coefficient of~$p_{nr}$ in the local
Legendre expansion of~$u$ on~$I_n$.  Since $a_{nr}(u)=O(k_n^r)$, the
result~\eqref{eq: error} shows that the dominant term in the DG error is
proportional to the Gauss--Radau polynomial~$p_{nr}(t)-p_{n,r-1}(t)$
for~$t\in I_n$.  However, the coefficient~$a_{nr}(u)$ and the $O(k_n^{r+1})$
term in~\eqref{eq: error} typically grow as~$t\to0$ at rates depending on the
regularity of the solution~$u$, which in turn depends on the regularity and
compatibility of the data. A possible extension permitting a time-dependent
operator~$A(t)$ is discussed briefly in~\cref{remark: A(t)}.

In 1985, Eriksson, Johnson and Thom\'ee~\cite{ErikssonEtAl1985} presented an
error analysis for DG time stepping of~\eqref{eq: IVP}, showing optimal
$O(k^{r+1})$~convergence in~$L_\infty\bigl((0,T);L_2(\Omega)\bigr)$ and
$O(k^{2r-1})$~superconvergence for the nodal
values~$\lim_{t\to t_n^-}U(t)$, where $k=\max_{1\le n\le N}k_n$.  Subsequently,
numerous authors~\cite{SchoetzauSchwab2000,ChrysafinosWalkington2006,
MakridakisNochetto2006,AkrivisEtAl2011,RichterEtAl2013,LeykekhmanVexler2017,
Saito2021} have refined these results, including a recent $L_\infty$~stability
result of Schmutz~and Wihler~\cite{SchmutzWihler2019} that we use in the proof
of \cref{thm: U - Pi tilde u}.  Shortly before completing the present work we
learned that the expansion~\eqref{eq: error} was proved by Adjerid et
al.~\cite{AdjeridEtAl2002,AdjeridBaccouch2010} for a linear, scalar hyperbolic
problem, and also for nonlinear systems of ODEs~\cite{Baccouch2016}; see
\cref{remark: Adjerid et al} for more details.

\cref{sec: consequences} discusses some practical consequences
of~\eqref{eq: error}, in particular the superconvergence of the DG solution at
the right Radau points in each interval. This phenomenon was exploited by
Springer and Vexler~\cite{SpringerVexler2014} in the piecewise-linear ($r=2$)
case to achieve higher-order accuracy for a parabolic optimal control problem.
We will see in \cref{lem: jump error} how the norm of the jump in~$U$ at the
break point~$t_{n-1}$ provides an accurate estimate of the maximum DG error over
the interval~$I_n$.  Moreover, a simple, low-cost post-processing step yields a
\emph{continuous} piecewise polynomial~$U_*$ of degree at most~$r$, called the
\emph{reconstruction} of~$U$, that satisfies $U_*(t)-u(t)=O(k_n^{r+1})$
for~$t\in I_n$; see \cref{cor: U* error}. Finally, \cref{sec: experiments}
reports the results of some numerical experiments for a scalar ODE and for heat
equations in one and two spatial dimensions, confirming the convergence
behaviour from the theory based on~\eqref{eq: error}.

Our motivation for the present study originated in a previous
work~\cite{McLean2020} dealing with the implementation of DG time stepping for
a subdiffusion equation
$u'(t)+\partial_t^{1-\nu}Au(t)=f(t)$ with $0<\nu<1$, where
$\partial_t^{1-\nu}$ denotes the Riemann--Liouville fractional time
derivative of order~$1-\nu$.  We observed in numerical experiments that
\eqref{eq: error} holds except with $O(k_n^{r+\nu})$ in place
of~$O(k_n^{r+1})$.

Treatment of the spatial discretization of~\eqref{eq: IVP} is beyond the scope
of this paper, apart from its use in our numerical experiments.  To make
practical use of our result~\eqref{eq: error} it is necessary to ensure
that the spatial error is dominated by the $O(k_n^{r+1})$~term.  Also, although
we allow nonuniform time steps in our analysis, we will not consider questions
such as local mesh refinement or adaptive step size control, which are
generally required to resolve the solution accurately for~$t$ near~$0$.
\section{Discontinuous Galerkin time stepping}\label{sec: DG}

As background and preparation for our results, we formulate in this section the
DG time stepping procedure and summarize key convergence results from the
literature.  Our standard reference is the monograph of
Thom\'ee~\cite[Chapter~12]{Thomee2006}.

Choosing time levels $0=t_0<t_1<t_2<\cdots<t_N=T$, we put
\[
k=\max_{1\le n\le N}k_n\quad\text{where}\quad k_n=t_n-t_{n-1}.
\]
Let $\mathbb{P}_j(\mathbb{V})$ denote the space of
polynomials of degree at most~$j$ with coefficients from a vector
space~$\mathbb{V}$.  We fix an integer~$r\ge1$, put
$\boldsymbol{t}=(t_n)_{n=0}^N$ and form the piecewise-polynomial
space~$\mathcal{X}_r=\mathcal{X}_r(\boldsymbol{t},\mathbb{H})$ defined by
\[
X\in\mathcal{X}_r\qquad\text{iff}\qquad
\text{$X\vert_{I_n}\in\mathbb{P}_{r-1}(\mathbb{H})$ for $1\le n \le N$.}
\]
Denoting the one-sided limits of~$X$ at~$t_n$ by
\[
X^n_+=\lim_{t\to t_n^+}X(t)\quad\text{and}\quad
X^n_-=\lim_{t\to t_n^-}X(t),
\]
we discretize \eqref{eq: IVP} in time by seeking $U\in\mathcal{X}_r$ satisfying
\cite[p.~204]{Thomee2006}
\begin{equation}\label{eq: DG}
\iprod{U^{n-1}_+,X^{n-1}_+}+\int_{I_n}\iprod{U'+AU,X}\,\ud t
    =\iprod{U^{n-1}_-,X^{n-1}_+}+\int_{I_n}\iprod{f,X}\,\ud t
\end{equation}
for $X\in\mathcal{X}_r$ and $1\le n\le N$, with $U^0_-=u_0$.
\cref{sec: Legendre} describes how, given $U^{n-1}_-$~and $f$, we can solve a
linear system to obtain~$U\vert_{I_n}$ and so advance the solution by one time
step.

\begin{remark}
If the integral on the right-hand side of~\eqref{eq: DG} is evaluated using the
right-hand, $r$-point, Gauss--Radau quadrature rule on~$I_n$, then the sequence
of nodal values~$U^n_-$ coincides with the finite difference solution produced
by the $r$-stage Radau IIA (fully) implicit Runge--Kutta method;
see Vlas\'ak and Roskovec~\cite[Section~3]{VlasakRoskovec2015}.
\end{remark}

Let $\|\cdot\|$ denote the norm in~$\mathbb{L}$ and let $u^{(\ell)}$ denote the
$\ell$th derivative of~$u$ with respect to~$t$. It will be convenient to write
\[
\|v\|_{I_n}=\sup_{t\in I_n}\|v(t)\|,
\]
and to define the fractional powers of~$A$ in the usual way via its
spectral decomposition~\cite[Chapter~3]{Thomee2006}. The DG time stepping
scheme has the nodal error bound~\cite[Theorem~12.1]{Thomee2006}
\begin{equation}\label{eq: nodal convergence}
\|U^n_--u(t_n)\|^2\le
C\sum_{j=1}^nk_j^{2\ell}\int_{I_j}\|A^{1/2}u^{(\ell)}(t)\|^2\,\ud t
\quad\text{for $1\le\ell\le r$,}
\end{equation}
and the uniform bound~\cite[Theorem~12.2]{Thomee2006}
\[
\|U-u\|_{I_n}\le\|U^n_--u(t_n)\|+C\|U^{n-1}_--u(t_{n-1})\|
    +Ck_n^\ell\|u^{(\ell)}\|_{I_n}\quad\text{for $1\le\ell\le r$,}
\]
where in both cases $1\le n\le N$. We therefore have optimal convergence
\begin{equation}\label{eq: U optimal}
\|U(t)-u(t)\|=O(k^r)\quad\text{for $0\le t\le T$,}
\end{equation}
provided $u^{(r)}\in L_\infty((0,T);\mathbb{L})$
and $A^{1/2}u^{(r)}\in L_2((0,T);\mathbb{L})$. In fact, $U$ is superconvergent
at the nodes~\cite[Theorem~12.3]{Thomee2006} when~$r\ge2$, with
\[
\|U^n_--u(t_n)\|^2\le Ck^{2(\ell-1)}\sum_{j=1}^nk_j^{2\ell}\int_{I_j}
    \|A^{\ell-1/2}u^{(\ell)}(t)\|^2\,\ud t\quad\text{for $1\le\ell\le r$.}
\]
Thus,
\begin{equation}\label{eq: nodal superconvergence}
\|U^n_--u(t_n)\|=O(k^{2r-1}),
\end{equation}
provided $A^{r-1/2}u^{(r)}\in L_2((0,T);\mathbb{L})$.

Suppose for the remainder of this section that $f\equiv0$, and consider error
bounds involving the (known) initial data~$u_0$ instead of the (unknown)
solution~$u$. By separating variables, one finds
that~\cite[Lemma~3.2]{Thomee2006}
\begin{equation}\label{eq: regularity f=0}
\|A^qu^{(\ell)}(t)\|\le Ct^{s-(q+\ell)}\|A^su_0\|
\quad\text{for $0\le s\le q+\ell$ and $0<t\le T$,}
\end{equation}
assuming that $u_0$ belongs to the domain of~$A^s$.  It follows that, for
sufficiently regular initial data, we have the basic error
bound~\cite[Theorem~1]{ErikssonEtAl1985},
\begin{equation}
\|U(t)-u(t)\|\le Ck^\ell\|A^\ell u_0\|
\quad\text{for $0\le t\le T$ and $0\le\ell\le r$.}
\end{equation}
For non-smooth initial data $u_0\in L_2(\Omega)$, the full rate of convergence
still holds but with a constant that blows up as~$t$ tends to
zero~\cite[Theorem~3]{ErikssonEtAl1985}: provided $k_n\le Ck_{n-1}$ for
all~$n\ge2$,
\[
\|U(t)-u(t)\|\le Ct^{-r}k^r\|u_0\|\quad\text{for $0<t\le T$,}
\]
and hence, by interpolation,
\begin{equation}\label{eq: U error s}
\|U(t)-u(t)\|\le Ct^{s-r}k^r\|A^su_0\|\quad
\text{for $0<t\le T$ and $0\le s\le r$.}
\end{equation}
At the nodes~\cite[Theorem~2]{ErikssonEtAl1985},
\begin{equation}\label{eq: nodal error smooth}
\|U^n_--u(t_n)\|\le Ck^s\|A^s u_0\|
\quad\text{for $1\le n\le N$ and $1\le s\le2r-1$,}
\end{equation}
and \cite[Theorem~3]{ErikssonEtAl1985}, provided $k_n\le Ck_{n-1}$ for
all~$n\ge2$,
\begin{equation}\label{eq: nodal error L2}
\|U^n_--u(t_n)\|\le Ct_n^{-s}k^s\|u_0\|
\quad\text{for $1\le n\le N$ and $0\le s\le2r-1$.}
\end{equation}
Taking $s=q$ in~\eqref{eq: nodal error smooth}
and~\eqref{eq: nodal error L2}, we see by interpolation that
\begin{equation}\label{eq: nodal error f=0}
\|U^n_--u(t_n)\|\le Ct_n^{s-q}k^q\|A^su_0\|
\quad\text{for $1\le n\le N$ and $0\le s\le q\le2r-1$.}
\end{equation}
\section{Local Legendre polynomial basis}\label{sec: Legendre}
We now return to considering the general inhomogeneous problem
and describe a practical formulation of the DG scheme using local Legendre
polynomial expansions that will also play an essential role in our subsequent
analysis.

Let $P_j$ denote the Legendre polynomial of degree~$j$ with the usual
normalization $P_j(1)=1$, and recall that
\[
\int_{-1}^1P_i(\tau)P_j(\tau)\,\ud\tau=\frac{2\delta_{ij}}{2j+1}
\quad\text{and}\quad
P_j(-\tau)=(-1)^jP_j(\tau).
\]
Using the affine map~$\beta_n:[-1,1]\to[t_{n-1},t_n]$ given by
\begin{equation}\label{eq: beta_n}
\beta_n(\tau)=\tfrac12[(1-\tau)t_{n-1}+(1+\tau)t_n]
\quad\text{for $-1\le\tau\le1$,}
\end{equation}
we define local Legendre polynomials on the $n$th subinterval,
\[
p_{nj}(t)=P_j(\tau)\quad\text{for $t=\beta_n(\tau)$ and $-1\le\tau\le1$,}
\]
and note that
\begin{equation}\label{eq: pnj properties}
p_{nj}(t_n)=1\quad\text{and}\quad
\int_{I_n}p_{ni}(t)p_{nj}(t)\,\ud t=\frac{k_n\delta_{ij}}{2j+1}.
\end{equation}
The local Fourier--Legendre expansion of a function~$v$ is then, for $t\in I_n$,
\[
v(t)=\sum_{j=0}^\infty a_{nj}(v)p_{nj}(t)
\quad\text{where}\quad
a_{nj}(v)=\frac{2j+1}{k_n}\int_{I_n}v(t)p_{nj}(t)\,\ud t.
\]

In particular, for the DG solution~$U$ we put $U^{nj}=a_{nj}(U)\in\mathbb{H}$
so that
\[
U(t)=\sum_{j=0}^{r-1}U^{nj}p_{nj}(t)\quad\text{for $t\in I_n$.}
\]
Define \cite[Lemma~5.1]{McLean2020}
\[
G_{ij}=P_j(-1)P_i(-1)+\int_{-1}^1P_j'(\tau)P_i(\tau)\,\ud\tau
    =\begin{cases}
   (-1)^{i+j},&\text{if $i\ge j$,}\\
    1,&\text{if $i<j$,}
\end{cases}
\]
and $H_{ij}=\int_{-1}^1P_j(\tau)P_i(\tau)\,d\tau=\delta_{ij}/(2j+1)$; e.g., if
$r=4$ then
\[
\boldsymbol{G}=\left[\begin{array}{rrrr}
 1& 1& 1&\phantom{-}1\\
-1& 1& 1& 1\\
 1&-1& 1& 1\\
-1& 1&-1& 1\end{array}\right]
\quad\text{and}\quad
\boldsymbol{H}=\begin{bmatrix}
1&        &        &        \\
 &\tfrac13&        &        \\
 &        &\tfrac15&        \\
 &        &        &\tfrac17\\
\end{bmatrix}.
\]
By choosing a test function of the form $X(t)=p_{ni}(t)\chi$,
for~$t\in I_n$ and $\chi\in\mathbb{H}$, we find that the DG
equation~\eqref{eq: DG} implies
\begin{equation}\label{eq: DG system}
\sum_{j=0}^{r-1}(G_{ij}+k_nH_{ij}A)U^{nj}
    =\check{U}^{n-1,i}+\int_{I_n}f(t)p_{ni}(t)\,\ud t
\end{equation}
for $0\le i\le r-1$ and $1\le n\le N$, where
\[
\check U^{0i}=(-1)^iu_0\quad\text{and}\quad
\check U^{ni}=(-1)^i\sum_{j=0}^{r-1}U^{nj}\quad\text{for $n\ge1$.}
\]
Thus, given $U^{n-1,j}$ for $0\le j\le r-1$, by solving the (block)
$r\times r$ system~\eqref{eq: DG system} we obtain $U^{nj}$ for~$0\le j\le r-1$,
and hence $U(t)$ for~$t\in I_n$.  The existence and uniqueness of this solution
follows from the stability of the scheme~\cite[p.~205]{Thomee2006}. Notice that
\[
U^{n-1}_+=\sum_{j=0}^{r-1}(-1)^jU^{nj}\quad\text{and}\quad
U^{n-1}_-=\begin{cases}
        u_0&\text{if $n=1$,}\\
        \sum_{j=0}^{r-1}U^{n-1,j}&\text{if $2\le n\le N$.}
\end{cases}
\]
\section{Behaviour of the DG error}\label{sec: Error behaviour}

To prove our main results, we will make use of two projection operators.
The first is just the orthogonal
projector~$\Pi_r:L^2((0,T);\mathbb{L})\to\mathcal{X}_r$ defined by
\[
\int_0^T\iprod{\Pi_rv,X}\,\ud t=\int_0^T\iprod{v,X}\,\ud t
    \quad\text{for all $X\in\mathcal{X}_r$,}
\]
which has the explicit representation
\[
(\Pi_rv)(t)=\sum_{j=0}^{r-1}a_{nj}(v)p_{nj}(t)
    \quad\text{for $t\in I_n$ and $1\le n\le N$.}
\]
The second projector~$\widetilde\Pi_r:C([0,T];\mathbb{L})\to\mathcal{X}_r$ is
defined by the conditions~\cite[Equation (12.9)]{Thomee2006}
\begin{equation}\label{eq: Pi tilde}
(\widetilde\Pi_r v)^n_-=v(t_n)\qquad\text{and}\qquad
\int_{I_n}\iprod{\widetilde\Pi_rv,X'}\,\ud t=\int_{I_n}\iprod{v,X'}\,\ud t
\end{equation}
for all $X\in\mathcal{X}_r$ and for $1\le n\le N$.  The next lemma
shows that $\widetilde\Pi_ru$ is in fact the DG solution of the trivial
equation with~$A=0$; cf.~Chrysafinos and
Walkington~\cite[Section~2.2]{ChrysafinosWalkington2006}.

\begin{lemma}\label{lem: Pi tilde A=0}
If $u':(0,T]\to\mathbb{L}$ is integrable, then
\[
\iprod{(\widetilde\Pi_ru)^{n-1}_+,X^{n-1}_+}
    +\int_{I_n}\iprod{(\widetilde\Pi_r u)',X}\,\ud t
    =\iprod{u(t_{n-1}),X^{n-1}_+}+\int_{I_n}\iprod{u',X}\,\ud t.
\]
\end{lemma}
\begin{proof}
Integrating by parts and using the properties~\eqref{eq: Pi tilde}
of~$\widetilde\Pi_r$, we have
\begin{align*}
\int_{I_n}\iprod{(\widetilde\Pi_r u)',X}\,\ud t
    &=\iprod{(\widetilde\Pi_r u)^n_-,X^n_-}
    -\iprod{(\widetilde\Pi_r u)^{n-1}_+,X^{n-1}_+}
    -\int_{I_n}\iprod{\widetilde\Pi_r u,X'}\,\ud t\\
    &=\iprod{u(t_n),X^n_-}
    -\iprod{(\widetilde\Pi_r u)^{n-1}_+,X^{n-1}_+}
    -\int_{I_n}\iprod{u,X'}\,\ud t,
\end{align*}
and a second integration by parts then yields the desired identity.
\end{proof}

The Legendre expansion of~$\widetilde\Pi_rv$ coincides with that of~$\Pi_rv$,
except for the coefficient of~$p_{n,r-1}$.  Below, we denote the closure of the
$n$th time interval by~$\bar I_n=[t_{n-1},t_n]$.

\begin{lemma}\label{lem: Pi tilde a tilde}
If $v:\bar I_n\to\mathbb{L}$ is continuous, then
\[
(\widetilde\Pi_rv)(t)=\sum_{j=0}^{r-2}a_{nj}(v)p_{nj}(t)
    +\tilde a_{n,r-1}(v)p_{n,r-1}(t)\quad\text{for $t\in I_n$,}
\]
where
\[
\tilde a_{n,r-1}(v)=v(t_n)-(\Pi_{r-1}v)^n_-
    =v(t_n)-\sum_{j=0}^{r-2}a_{nj}(v).
\]
\end{lemma}
\begin{proof}
By choosing $X'\vert_{I_n}=p_{nj}$ in the second property
of~\eqref{eq: Pi tilde}, we see that
\[
a_{nj}(\widetilde\Pi_rv)=a_{nj}(v)\quad\text{for $0\le j\le r-2$,}
\]
implying that $\widetilde\Pi_rv=\Pi_{r-1}v+\lambda p_{n,r-1}$ for
some $\lambda\in\mathbb{H}$. Since $p_{nj}(t_n)=P_j(1)=1$, the first property
in~\eqref{eq: Pi tilde} gives
\[
v(t_n)=(\widetilde\Pi_rv)^n_-=(\Pi_{r-1}v)^n_-+\lambda\quad\text{with}\quad
(\Pi_{r-1}v)^n_-=\sum_{j=0}^{r-2}a_{nj}(v),
\]
showing that $\lambda=\tilde a_{n,r-1}(v)$.
\end{proof}

By mapping to the reference element~$(-1,1)$, applying the Peano kernel
theorem, and then mapping back to~$I_n$, we find~\cite[p.~137]{McLean2020}
\begin{equation}\label{eq: anj estimate}
\|a_{nj}(v)\|\le Ck_n^{j-1}\int_{I_n}\|v^{(j)}(t)\|\,\ud t
    \le Ck_n^j\|v^{(j)}\|_{I_n}\quad\text{for $j\ge0$,}
\end{equation}
and
\begin{equation}\label{eq: v - Pi_r v}
\|v-\Pi_rv\|_{I_n}\le Ck_n^{\ell-1}\int_{I_n}\|v^{(\ell)}(t)\|\,\ud t
    \le Ck_n^\ell\|v^{(\ell)}\|_{I_n}\quad\text{for $1\le\ell\le r$.}
\end{equation}

\begin{theorem}\label{thm: Pi tilde error}
For $1\le n\le N$, if $v:\bar I_n\to\mathbb{L}$ is $C^{r+1}$ then
\[
\bigl\|\widetilde\Pi_rv-v+a_{nr}(v)(p_{nr}-p_{n,r-1})\bigr\|_{I_n}
    \le Ck_n^{r+1}\|v^{(r+1)}\|_{I_n}.
\]
\end{theorem}
\begin{proof}
By \cref{lem: Pi tilde a tilde}, if $t\in I_n$ then
\[
(\widetilde\Pi_rv)(t)=(\Pi_{r-1}v)(t)+\tilde a_{n,r-1}(v)p_{n,r-1}(t)
\]
and
\[
(\Pi_{r+1}v)(t)=(\Pi_{r-1}v)(t)+a_{n,r-1}(v)p_{n,r-1}(t)+a_{nr}(v)p_{nr}(t),
\]
so
\begin{equation}\label{eq: Pi tilde Pi}
(\widetilde\Pi_rv)(t)-(\Pi_{r+1}v)(t)
    =[\tilde a_{n,r-1}(v)-a_{n,r-1}(v)]p_{n,r-1}(t)-a_{nr}(v)p_{nr}(t).
\end{equation}
Taking the limit as $t\to t_n^-$, and recalling that $p_{nj}(t_n)=1$, we see
that
\begin{equation}\label{eq: an tilde an}
v(t_n)-(\Pi_{r+1}v)^n_-=\tilde a_{n,r-1}(v)-a_{n,r-1}(v)-a_{nr}(v).
\end{equation}
Using \eqref{eq: an tilde an} to eliminate $\tilde a_{n,r-1}(v)$
in~\eqref{eq: Pi tilde Pi}, we find that
\[
\widetilde\Pi_rv-v+a_{nr}(v)[p_{nr}-p_{n,r-1}]=
(\Pi_{r+1}v-v)+[v(t_n)-(\Pi_{r+1}v)^n_-]p_{n,r-1}
\]
on $I_n$, and the desired estimate follows at once from~\eqref{eq: v - Pi_r v}.
\end{proof}

The following theorem and its corollary, together with the superconvergence
result \eqref{eq: nodal superconvergence}, show that
\begin{equation}\label{eq: super-approx}
\|U-\widetilde\Pi_ru\|_{I_n}=O(k_n^{r+1})\quad\text{for $r\ge2$,}
\end{equation}
provided $u$ is sufficiently regular.

\begin{theorem}\label{thm: U - Pi tilde u}
For $1\le n\le N$, if $Au:\bar I_n\to\mathbb{L}$ is $C^r$, then
\[
\|U-\widetilde\Pi_ru\|_{I_n}\le C\|U^{n-1}_--u(t_{n-1})\|
    +Ck_n^{r+1}\|Au^{(r)}\|_{I_n}.
\]
\end{theorem}
\begin{proof}
It follows from \cref{lem: Pi tilde A=0} that $\widetilde\Pi_ru$ satisfies
\begin{multline*}
\iprod{(\widetilde\Pi_ru)^{n-1}_+,X^{n-1}_+}
    +\int_{I_n}\iprod{(\widetilde\Pi_r u)'+A\widetilde\Pi_r u,X}\,\ud t\\
    =\iprod{u(t_{n-1}),X^{n-1}_+}
    +\int_{I_n}\iprod{u'+A\widetilde\Pi_r u,X}\,\ud t,
\end{multline*}
whereas $U$ satisfies
\[
\iprod{U^{n-1}_+,X^{n-1}_+}+\int_{I_n}\iprod{U'+AU,X}\,\ud t
    =\iprod{U^{n-1}_-,X^{n-1}_+}+\int_{I_n}\iprod{u'+Au,X}\,\ud t,
\]
for all $X\in\mathcal{X}_r$.  Letting $\rho=A(u-\widetilde\Pi_r u)$ and noting
$(\widetilde\Pi_r u)^{n-1}_-=u(t_{n-1})$, we see that the piecewise
polynomial~$\varepsilon=U-\widetilde\Pi_r u\in\mathcal{X}_r$ satisfies
\begin{equation}\label{eq: epsilon}
\iprod{\varepsilon^{n-1}_+,X^{n-1}_+}
    +\int_{I_n}\iprod{\varepsilon'+A\varepsilon,X}\,\ud t
    =\iprod{\varepsilon^{n-1}_-,X^{n-1}_+}+\int_{I_n}\iprod{\rho,X}\,\ud t
\end{equation}
for all $X\in\mathcal{X}_r$, with $\varepsilon^{n-1}_-=U^{n-1}_--u(t_{n-1})$.
A stability result of Schmutz and
Wihler~\cite[Proposition~3.18]{SchmutzWihler2019} yields the estimate
\begin{equation}\label{eq: stable L infty}
\|\varepsilon\|_{I_n}^2\le C\biggl(\|\varepsilon^{n-1}_-\|^2
    +k_n\int_{I_n}\|\rho\|^2\,\ud t\biggr),
\end{equation}
that is,
\[
\|U-\widetilde\Pi_r u\|_{I_n}^2\le C\biggl(\|U^{n-1}_--u(t_{n-1})\|^2
    +k_n\int_{I_n}\|\rho\|^2\,\ud t\biggr).
\]
By putting $v=Au$ in~\eqref{eq: v - Pi_r v} we find
$k_n\int_{I_n}\|\rho\|^2\,\ud t\le k_n^2\|\rho\|_{I_n}^2\le
C(k_n^{r+1}\|Au^{(r)}\|_{I_n})^2$, and the
desired estimate follows at once.
\end{proof}

We are now able to establish the claim~\eqref{eq: error} from the Introduction.

\begin{theorem}\label{thm: main result}
For $1\le n\le N$, if $Au^{(r)}$ and $u^{(r+1)}$ are continuous on~$\bar I_n$,
then
\begin{multline*}
\|U-u+a_{nr}(u)(p_{nr}-p_{n,r-1})\|_{I_n}\le C\|U^{n-1}_--u(t_{n-1})\|\\
    +Ck_n^{r+1}\bigl(\|Au^{(r)}\|_{I_n}+\|u^{(r+1)}\|_{I_n}\bigr).
\end{multline*}
\end{theorem}
\begin{proof}
Write
\[
U-u+a_{nr}(u)(p_{nr}-p_{n,r-1})=(U-\widetilde\Pi_ru)+\bigl(\widetilde\Pi_ru-u
+a_{nr}(u)(p_{nr}-p_{n,r-1})\bigr),
\]
and apply \cref{thm: Pi tilde error,thm: U - Pi tilde u}.
\end{proof}

We therefore have the following estimate for the homogeneous problem expressed
in terms of the initial data.

\begin{corollary}\label{cor: main result}
Assume $k_n\le Ck_{n-1}$ for~$2\le n\le N$ so that \eqref{eq: nodal error f=0}
holds. If $f\equiv0$, then for $0\le s\le r+1$~and $2\le n\le N$,
\[
\|U-u+a_{nr}(u)(p_{nr}-p_{n,r-1})\|_{I_n}\le C
    t_n^{s-(r+1)}k^{r+1}\|A^su_0\|.
\]
\end{corollary}
\begin{proof}
Taking $q=r+1$ in~\eqref{eq: nodal error f=0} yields
\[
\|U^{n-1}_--u(t_{n-1})\|\le C
    t_{n-1}^{s-(r+1)}k^{r+1}\|A^su_0\|,
\]
and using \eqref{eq: regularity f=0} we have
$\|Au^{(r)}(t)\|=\|u^{(r+1)}(t)\|\le Ct^{s-(r+1)}\|A^su_0\|$.  The result
follows for~$n\ge2$ after noting that
$t_n=t_{n-1}+k_n\le t_{n-1}+Ck_{n-1}\le Ct_{n-1}$.
\end{proof}

\begin{remark}\label{remark: Adjerid et al}
In their proof of~\eqref{eq: error} for the scalar linear problem
\[
u'-au=0\quad\text{for $t>0$, with $u(0)=u_0$,}
\]
Adjerid et al.~\cite[Theorem~3]{AdjeridEtAl2002} use an inductive argument
to show an expansion of the form
\[
U(t)-u(t)=\sum_{j=r}^{2r-2}Q_{nj}(t)\,k_n^j+O(k_n^{2r-1})\quad
\text{for $t\in I_n$,}
\]
where $Q_{nj}\in\mathbb{P}_{j-1}$ and
$Q_{nr}(t)=c_{np}[p_{nr}(t)-p_{n,r-1}(t)]$ for a constant~$c_{np}$.
They extend this result to a homogeneous linear system of
ODEs~$\bs{u}'-\bs{A}\bs{u}=\bs{0}$, then a nonlinear scalar
problem~$u'-f(u)=0$, and finally a nonlinear
system~$\bs{u}'-\bs{f}(\bs{u})=\bs{0}$.
\end{remark}

\begin{remark}\label{remark: A(t)}
The proof of \cref{thm: U - Pi tilde u} is largely unaffected if the elliptic
term is permitted to have time-dependent coefficients, resulting in a
time-dependent operator~$A(t)$. The main issue is to verify the stability
property~\eqref{eq: stable L infty} for this more general setting.  The only
other complication is the estimation of~$\rho(t)$. Consider, for example,
$A(t)u(x,t)=-\nabla\cdot\bigl(a(x,t)\nabla u(x,t)\bigr)$.
Since $A(t)u(x,t)$ is of the form~$\sum_{m=1}^M c_m(x,t)B_mu(x,t)$, where each
$B_m$ is a second-order linear differential operator involving only the spatial
variables~$x$, it follows that
\[
\rho(t)=A(t)\bigl(u(t)-\widetilde{\Pi}_ru(t)\bigr)
    =\sum_{m=1}^Mc_m(x,t)\bigl(B_mu(t)-\widetilde{\Pi}_rB_mu(t)\bigr),
\]
and the final step of the proof becomes
\[
k_n\int_{I_n}\|\rho\|^2\,\ud t\le
Ck_n^{2(r+1)}\sum_{m=1}^M\|B_mu^{(r)}\|_{I_n}^2.
\]
Of course, to exploit this generalization of \cref{thm: U - Pi tilde u}, it
would also be necessary to verify the superconvergent error bounds for~$U^n_-$
in this case.
\end{remark}

\begin{figure}
\centering
\includegraphics[scale=0.6]{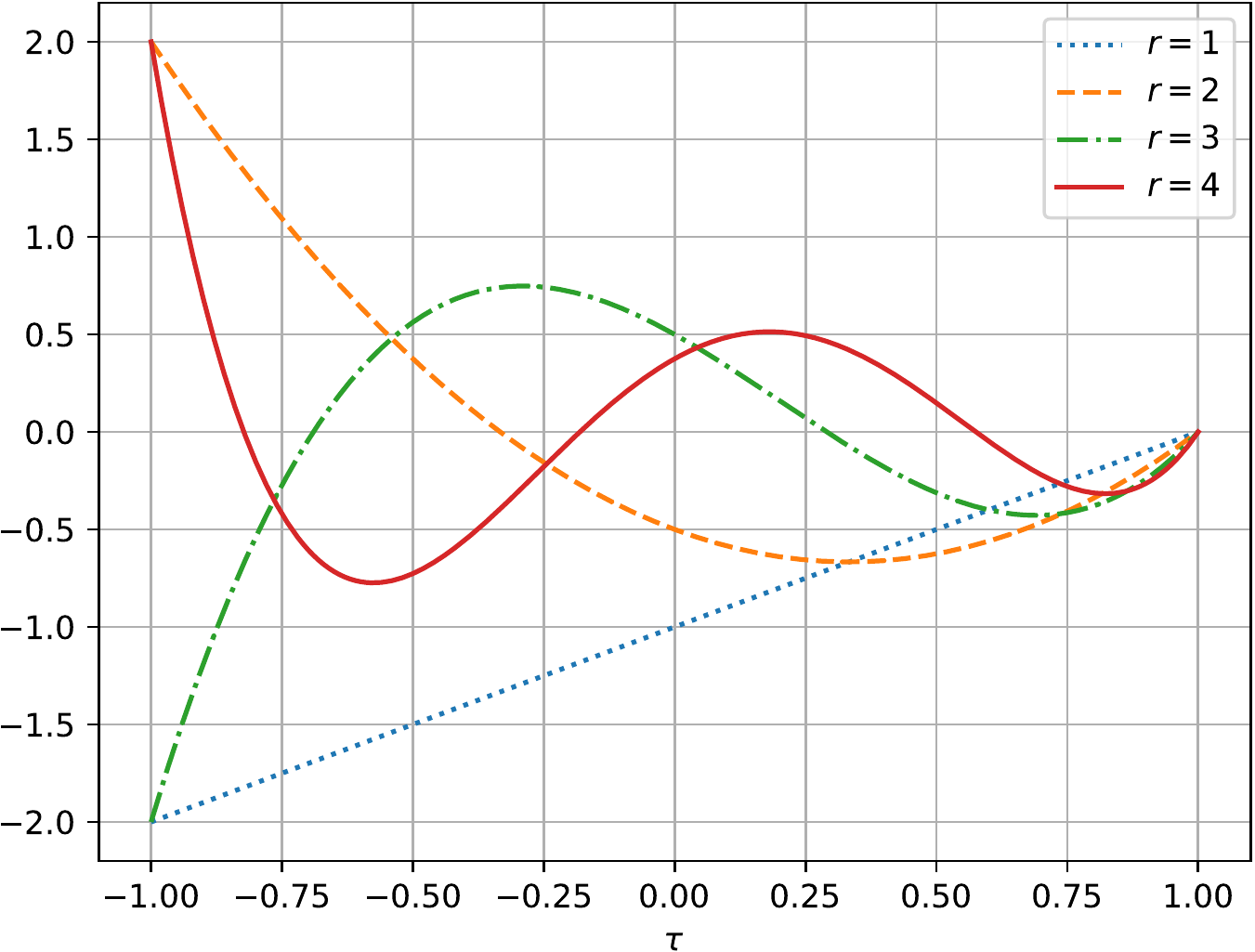}
\caption{The polynomials $P_r(\tau)-P_{r-1}(\tau)$.}\label{fig: Radau polys}
\end{figure}

\section{Practical consequences}\label{sec: consequences}
Throughout this section, we will assume that
\begin{equation}\label{eq: phi assumption}
\|U^{n-1}_--u(t_{n-1})\|+
\|U-u+a_{nr}(u)(p_{nr}-p_{n,r-1})\|_{I_n}\le C\phi(t_n,u)k_n^{r+1},
\end{equation}
for~$2\le n\le N$, where the factor~$\phi(t,u)$ will depend on the
regularity of~$u$, which in turn depends on the regularity and compatibility of
the initial data~$u_0$ and the source term~$f$. \cref{fig: Radau polys} plots
the right-hand Gauss--Radau polynomials
\[
p_{nr}(t)-p_{n,r-1}(t)=P_r(\tau)-P_{r-1}(\tau)
\]
as functions of~$\tau\in[-1,1]$ for $r\in\{1,2,3,4\}$.  In general, there are
$r+1$ points
\[
-1=\tau_0<\tau_1<\cdots<\tau_r=1,
\]
such that $\tau_1$, $\tau_2$, \dots, $\tau_r$ are the $r$~zeros of
$P_r-P_{r-1}$, and hence are also the abscissas of the right-hand, 
$r$-point Gauss--Radau quadrature rule for the interval~$[-1,1]$.
Recalling our previous notation~\eqref{eq: beta_n}, let
$t_{n\ell}=\beta_n(\tau_\ell)$ so that
$t_{n-1}=t_{n0}<t_{n1}<\cdots<t_{nr}=t_n$ with
\[
p_{nr}(t_{n\ell})-p_{n,r-1}(t_{n\ell})=0\quad\text{for $1\le\ell\le r$.}
\]
Thus, whereas $U(t)-u(t)=O(k_n^r)$ for general~$t\in I_n$, the DG time stepping
scheme is superconvergent at the $r$ special points $t_{n1}$, $t_{n2}$,
\dots, $t_{nr}$ in the half-open interval~$(t_{n-1},t_n]$.  More precisely,
\[
\|U(t_{n\ell})-u(t_{n\ell})\|\le C\phi(t_n,u)k_n^{r+1}
    \quad\text{for $1\le\ell\le r$.}
\]

Since $p_{nj}(t_{n-1})=P_j(-1)=(-1)^j$, another consequence 
of~\eqref{eq: phi assumption} is that
\[
\|U^{n-1}_+-u(t_{n-1})+2(-1)^ra_{nr}(u)\|\le C\phi(t_n,u)k_n^{r+1},
\]
which, in combination with the estimate
$\|U^{n-1}_--u(t_{n-1})\|\le C\phi(t_n,u)k_n^{r+1}$, shows that the
jump~$\jump{U}^{n-1}=U^{n-1}_+-U^{n-1}_-$ in the DG solution at~$t_{n-1}$
satisfies
\begin{equation}\label{eq: jump anr}
\bigl\|\jump{U}^{n-1}+2(-1)^r a_{nr}(u)\bigr\|\le C\phi(t_n,u)k_n^{r+1}.
\end{equation}
We are therefore able to show, in the following lemma, that
$\|\jump{U}^{n-1}\|$ is a low-cost and accurate error indicator for the DG
solution on~$I_n$.

\begin{lemma}\label{lem: jump error}
For $\phi$ as in~\eqref{eq: phi assumption} and $2\le n\le N$,
\[
\bigl\vert\|U-u\|_{I_n}-\|\jump{U}^{n-1}\|\bigr\vert
		\le C\phi(t_n,u)k_n^{r+1}.
\]
Thus,
\[
\|U-u\|_{I_n}=2\|a_{nr}(u)\|+O(k_n^{r+1})
    =\bigl\|\jump{U}^{n-1}\bigr\|+O(k_n^{r+1}).
\]
\end{lemma}
\begin{proof}
First note that since
\[
\max_{-1\le\tau\le1}\vert P_r(\tau)-P_{r-1}(\tau)\vert
=\vert P_r(-1)-P_{r-1}(-1)\vert=2,
\]
we have
\begin{equation}\label{eq: max Radau poly}
\|a_{nr}(u)(p_{nr}-p_{n,r-1})\|_{I_n}
    =\vert p_{nr}(t_{n-1})-p_{n,r-1}(t_{n-1})\vert\|a_{nr}(u)\|=2\|a_{nr}(u)\|.
\end{equation}
Hence, for~$t\in I_n$,
\begin{align*}
\|U(t)-u(t)\|&\le\|U(t)-u(t)+a_{nr}(u)[p_{nr}(t)-p_{n,r-1}(t)]\|\\
    &\qquad{}+\|a_{nr}[p_{nr}(t)-p_{n,r-1}(t)]\|
    \le C\phi(t_n,u)k_n^{r+1}+2\|a_{nr}(u)\|,
\end{align*}
and so $\|U-u\|_{I_n}\le2\|a_{nr}(u)\|+C\phi(t_n,u)k_n^{r+1}$.  Conversely,
\begin{align*}
2\|a_{nr}(u)\|&=\|a_{nr}(u)[p_{nr}(t_{n-1})-p_{n,r-1}(t_{n-1})]\|\\
    &\le\|U^{n-1}_+-u(t_{n-1})+a_{nr}(u)[p_{nr}(t_{n-1})-p_{n,r-1}(t_{n-1})]\|\\
    &\qquad{}+\|U^{n-1}_+-u(t_{n-1})\|\\
    &\le C\phi(t_n,u)k_n^{r+1}+\|U-u\|_{I_n},
\end{align*}
and therefore
\[
\bigl\vert\|U-u\|_{I_n}-2\|a_{nr}(u)\|\bigr\vert\le C\phi(t_n,u)k_n^{r+1}.
\]
Since, by~\eqref{eq: jump anr},
\begin{align*}
\bigl\vert\|\jump{U}^{n-1}\|-2\|a_{nr}(u)\|\bigr\vert
    &=\bigl\vert\|\jump{U}^{n-1}\|-\|2(-1)^{r+1}a_{nr}(u)\|\bigr\vert\\
    &\le\bigl\|\jump{U}+2(-1)^ra_{nr}(u)\bigr\|\le C\phi(t_n,u)k_n^{r+1},
\end{align*}
the result follows.
\end{proof}

A unique continuous function~$U_*\in\mathcal{X}_{r+1}$ satisfies the
$r+1$~interpolation conditions
\[
U_*(t_{n\ell})=\begin{cases}
U^{n-1}_-&\text{if $\ell=0$,}\\
U(t_{n\ell})&\text{if $1\le\ell\le r-1$,}\\
U^n_-&\text{if $\ell=r$,}
\end{cases}
\]
for $1\le n\le N$, and we see that
\begin{equation}\label{eq: U* interp error}
(U_*-u)(t_{n\ell})=O(k_n^{r+1})\quad\text{for $0\le\ell\le r$.}
\end{equation}
Makridakis and Nochetto~\cite{MakridakisNochetto2006} introduced this
interpolant in connection with \emph{a posteriori} error analysis of diffusion
problems, and called $U_*$ the \emph{reconstruction} of~$U$.  The next theorem
provides a more explicit description of~$U_*$ that we then use to prove $U_*$
achieves the optimal convergence rate of order~$k_n^{r+1}$ over the whole
subinterval~$I_n$.

\begin{theorem}\label{thm: error estimator}
For $t\in I_n$ and $1\le n\le N$, the reconstruction~$U_*$ of the DG
solution~$U$ has the representation
\[
U_*(t)=U(t)-\frac{(-1)^r}{2}\jump{U}^{n-1}(p_{nr}-p_{n,r-1})(t)
    =\sum_{j=0}^r U_*^{nj}p_{nj}(t),
\]
where
\[
U_*^{nj}=\begin{cases}
    U^{nj}&\text{if $0\le j\le r-2$,}\\
    U^{n,r-1}+\tfrac12(-1)^r\jump{U}^{n-1}&\text{if $j=r-1$,}\\
    -\tfrac12(-1)^r\jump{U}^{n-1}&\text{if $j=r$.}
\end{cases}
\]
\end{theorem}
\begin{proof}
Since the polynomial~$(U-U_*)\vert_{I_n}\in\mathbb{P}_r(\mathbb{H})$ vanishes
at~$t_{n\ell}$ for $1\le\ell\le r$, there must be a constant~$\gamma$ such that
$U(t)-U_*(t)=\gamma(p_{nr}-p_{n,r-1})(t)$ for~$t\in I_n$.  Taking the limit
as~$t\to t_{n-1}^+$, we have $U^{n-1}_+-U^{n-1}_-=\gamma[(-1)^r-(-1)^{r-1}]
=2(-1)^r\gamma$ and so $\gamma=(-1)^r\jump{U}^{n-1}/2$. It follows
from~\eqref{eq: pnj properties} that
\[
a_{nj}(U-U_*)=\frac{2j+1}{k_n}\int_{I_n}(U-U_*)(t)p_{nj}(t)\,\ud t
    =\frac{(-1)^r}{2}\jump{U}^{n-1}(\delta_{jr}-\delta_{j,r-1}),
\]
implying the formulae for $U^{nj}_*=a_{nj}(U_*)$.
\end{proof}
\begin{corollary}\label{cor: U* error}
$\|U_*-u\|_{I_n}\le C\phi(t_n,u)k_n^{r+1}$ for $2\le n\le N$.
\end{corollary}
\begin{proof}
We see from the \cref{thm: error estimator} and \eqref{eq: max Radau poly}
that
\begin{align*}
\|U_*-u\|_{I_n}&=\|U-u-\tfrac12(-1)^r\jump{U}^{n-1}(p_{nr}-p_{n,r-1})\|_{I_n}\\
    &\le\|U-u+a_{nr}(u)(p_{nr}-p_{n,r-1})\|_{I_n}
    +\tfrac12\|\jump{U}^{n-1}+2(-1)^ra_{nr}(u)\|,
\end{align*}
so it suffices to apply \eqref{eq: phi assumption}~and \eqref{eq: jump anr}.
\end{proof}

\begin{example}\label{example: weighted error f=0}
Let $f\equiv0$ and let $u_0$ belong to the domain of~$A^s$.
By~\eqref{eq: U error s},
\[
t^{r-s}\|U(t)-u(t)\|\le Ck^r\|A^su_0\|
\quad\text{if $0<t\le T$ and $0\le s\le r$,}
\]
and by \eqref{eq: nodal error f=0},
\[
t_n^{2r-1-s}\|U^n_--u(t_n)\|\le C
    k^{2r-1}\|A^su_0\|\quad\text{if $1\le n\le N$ and $0\le s\le 2r-1$.}
\]
Furthermore, \cref{cor: main result} shows that our assumption~\eqref{eq: phi
assumption} is satisfied with
\[
\phi(t,u)=t^{s-(r+1)}\|A^su_0\|
\]
so
\[
t_n^{r+1-s}\|U_*-u\|_{I_n}\le Ck^{r+1}\|A^su_0\|
\quad\text{if $2\le n\le N$ and $0\le s\le r+1$.}
\]
\end{example}
\section{Numerical experiments}\label{sec: experiments}
The computational experiments described in this section were performed in
standard 64-bit floating point arithmetic using Julia v1.7.2 on a desktop
computer having a Ryzen~7 3700X processor and $32\,\mathrm{GiB}$ of RAM. The
source code is available online~\cite{McLean2022}.  In all cases, we use
uniform time steps $k_n=k=T/N$.

\subsection{A simple ODE.}\label{sec: ODE}
We begin with the ODE initial-value problem
\[
u'+\lambda u=f(t)\quad\text{for $0\le t\le 2$, with $u(0)=1$,}
\]
where in place of a linear operator~$A$ we have just the scalar~$\lambda=1/2$,
and where $f(t)=\cos(\pi t)$.  For the piecewise-cubic case with~$N=5$
subintervals, \cref{fig: ODE} shows that $U-U_*$ provides an excellent
approximation to the error~$U-u$, and that the error profile is approximately
proportional to $p_{nr}-p_{n,r-1}$ with~$r=4$; cf.~\eqref{eq:
super-approx}~and \cref{fig: Radau polys}. In particular, superconvergence at
the Radau points is apparent.  By sampling at 50~points in each subinterval, we
estimated the maximum errors
\[
\adjustlimits\max_{1\le n\le N}\sup_{t\in I_n}\vert U(t)-u(t)\vert
\quad\text{and}\quad
\adjustlimits\max_{1\le n\le N}\sup_{t\in I_n}\vert U_*(t)-u(t)\vert,
\]
and, as expected from \eqref{eq: U optimal}~and \cref{cor: U* error}, the
values shown in \cref{table: ODE} exhibit convergence rates
$r=4$~and $r+1=5$, respectively.  The table also shows a convergence
rate~$2r-1=7$ for the nodal error
$\max_{1\le n\le N}\vert U^n_--u(t_n)\vert$ up
to the row where this error approaches the unit roundoff.  By using Julia's
BigFloat datatype, we were able to observe $O(k^7)$ convergence of~$U^n_-$
up to~$N=128$, for which value the nodal error was \texttt{1.56e-19}.

\begin{figure}
\begin{center}
\includegraphics[scale=0.6]{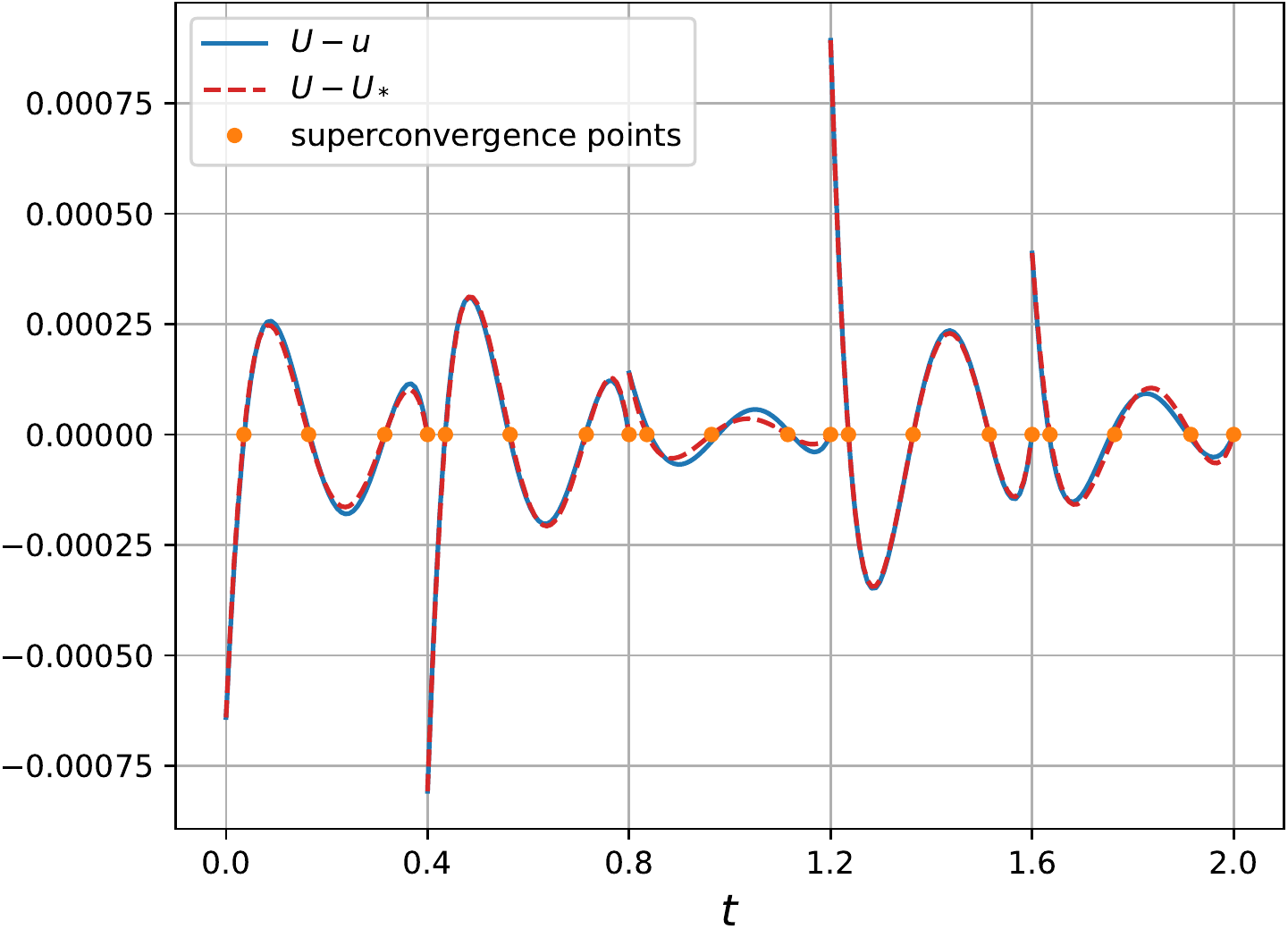}
\end{center}
\caption{The DG error $U-u$, the difference $U-U_*$ between the DG solution and
its reconstruction, along with the superconvergence points~$t_{nj}$
($1\le j\le r$), for the ODE of \cref{sec: ODE} using piecewise-cubics
($r=4$).}\label{fig: ODE}
\end{figure}

\begin{table}
\caption{Errors and convergence rates for the ODE of \cref{sec: ODE} using
piecewise-cubics ($r=4$).}\label{table: ODE}
\renewcommand{\arraystretch}{1.2}
\begin{center}
\ttfamily
\begin{tabular}{r|rr|rr|rr}
$N$&\multicolumn{2}{c|}{\textrm{Error in $U$}}&
\multicolumn{2}{c|}{\textrm{Error in $U_*$}}&
\multicolumn{2}{c}{\textrm{Error in $U^n_-$}}\\
\hline
   4 &  1.75e-03 &         &  6.15e-05 &         &  5.26e-09\\
   8 &  1.36e-04 &   3.684 &  2.26e-06 &   4.769 &  4.08e-11 &   7.010\\
  16 &  8.85e-06 &   3.945 &  7.19e-08 &   4.973 &  3.27e-13 &   6.962\\
  32 &  5.55e-07 &   3.996 &  2.26e-09 &   4.994 &  2.66e-15 &   6.941\\
  64 &  3.48e-08 &   3.995 &  7.05e-11 &   5.000 &  7.77e-16 &   1.778\\
 128 &  2.17e-09 &   3.999 &  2.20e-12 &   4.999 &  1.55e-15 &  -1.000\\
\hline
\textrm{Theory}&          &      $4$&           &      $5$&           & $7$
\end{tabular}
\end{center}
\end{table}

\subsection{A parabolic PDE in 1D}\label{sec: PDE 1D}
Now consider the 1D heat equation with constant thermal conductivity~$\kappa>0$,
\begin{equation}\label{eq: heat 1D}
u_t-\kappa u_{xx}=f(x,t)\quad\text{for $0<t\le T$ and $0\le x\le L$,}
\end{equation}
subject to the boundary conditions $u(0,t)=0=u(L,t)$ for $0\le t\le T$, and to
the initial condition $u(x,0)=u_0(x)$ for $0\le x\le L$.  To obtain a
reference solution, we introduce the Laplace transform
$\hat u(x,z)=\int_0^\infty e^{-zt}u(x,t)\,\ud t$, which satisfies the two-point
boundary-value problem (with complex parameter~$z$),
\[
-\hat u_{xx}+\omega^2\hat u=g(x,z)\quad\text{for $0\le x\le L$,}
\quad\text{with $\hat u(0,z)=0=\hat u(L,z)$,}
\]
where $\omega=(z/\kappa)^{1/2}$ and $g(x,z)=\kappa^{-1}[u_0(x)+\hat f(x,z)]$.
Consequently, the variation-of-constants formula yields the
representation~\cite[Section~7.3]{McLean2020}
\begin{multline*}
\hat u(x,z)=\frac{\sinh\omega(L-x)}{\omega\sinh\omega L}\int_0^x
    g(\xi,z)\sinh\omega\xi\,\ud\xi\\
    +\frac{\sinh\omega x}{\omega\sinh\omega L}\int_x^L
    g(\xi,z)\sinh\omega(L-\xi)\,\ud\xi,
\end{multline*}
and we then invert the Laplace transform by numerical evaluation of the
Bromwich integral~\cite{WeidemanTrefethen2007},
\[
u(x,t)=\frac{1}{2\pi i}\int_{\mathcal{C}}e^{zt}\hat u(x,z)\,dz,
\]
for a hyperbolic contour~$\mathcal{C}$ homotopic to the imaginary axis
and passing to the right of all singularities of~$\hat u(x,z)$.

\begin{figure}
\begin{center}
\includegraphics[scale=0.6]{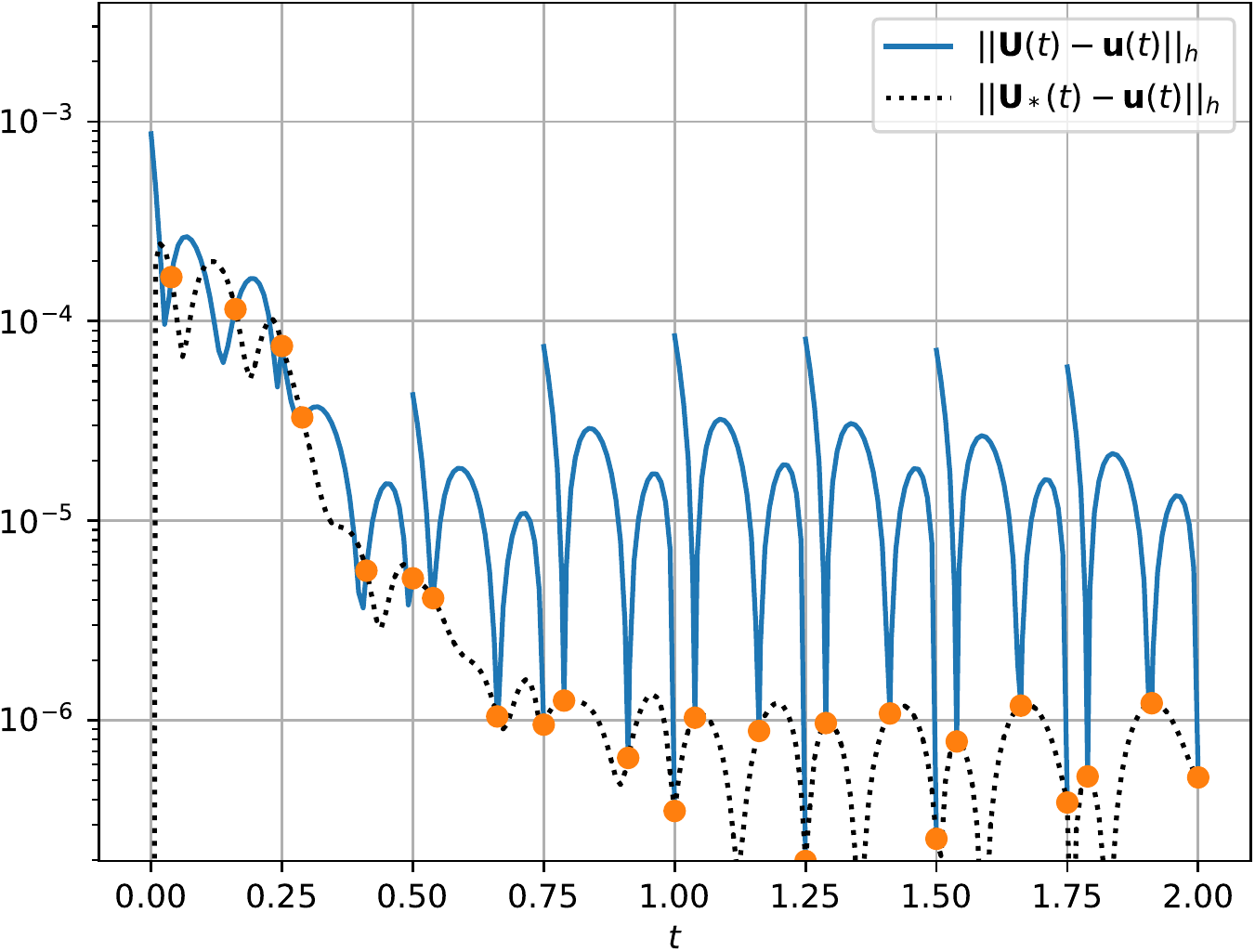}
\end{center}
\caption{Time dependence of the errors in the DG solution~$U(t)$ and its
reconstruction~$U_*(t)$ for the 1D heat equation~\eqref{eq: heat 1D}, using
piecewise-quadratics ($r=3$) over $N=8$ time intervals.}
\label{fig: errors 1D}
\end{figure}

To discretise in space, we introduce a finite difference grid
\[
x_p=p\,h\quad\text{for $0\le p\le P$,}\quad\text{where $h=L/P$,}
\]
and define $u_p(t)\approx u(x_p,t)$ via the method of lines, replacing
$u_{xx}$ with a second-order central difference approximation to arrive at
the system of ODEs
\begin{equation}\label{eq: MOL 1D}
u_p'(t)-\kappa\frac{u_{p+1}(t)-2u_p(t)+u_{p-1}(t)}{h^2}=f_p(t)
\quad\text{for $1\le p\le P-1$,}
\end{equation}
where $f_p(t)=f(x_p,t)$ with the boundary conditions
$u_0(t)=0=u_P(t)$ and the initial condition~$u_p(0)=u_0(x_p)$.  For our
test problem, we choose
\begin{equation}\label{eq: PDE1D data}
L=2,\quad T=2,\quad\kappa=(L/\pi)^2,\quad u_0(x)=x(L-x),\quad
f(x,t)=(1+t)e^{-t},
\end{equation}
where the value of the thermal conductivity~$\kappa$ normalises the time
scale by making the smallest eigenvalue of~$A=-\kappa(d/dx)^2$ equal~$1$.  We
will see below that $u_0\in D(A^s)$ iff $s<5/4$, so the regularity of the
solution~$u$ is limited.

We apply DG to discretise $u_p(t)$ in time and denote the resulting
fully-discrete solution by~$\bs{U}(t)=[U_p(t)]\approx\bs{u}(t)=[u_p(t)]$.
\cref{fig: errors 1D} plots the error in $\bs{U}$ and in its
reconstruction~$\bs{U}_*$ using piecewise-quadratics ($r=3$) and $N=8$~equal
subintervals in time, with $P=500$ for the spatial grid. The errors are
measured in the discrete $L_2$-norm, that is,
\[
\|\bs{U}(t)-\bs{u}(t)\|_h^2=\sum_{p=0}^P\vert U_p(t)-u(x_p,t)\vert^2\,h,
\]
and we observe a clear deterioration in accuracy as~$t$ approaches zero.

To speed up the convergence as~$h\to0$, we compute also a second DG
solution~$U^\fine_p(t)$ using a finer spatial grid with~$P^\fine=2P$
subintervals, and then perform one step of Richardson extrapolation (on the
coarser grid), defining
\[
\UR_p(t)=U^\fine_{2p}(t)+\tfrac13[U^\fine_{2p}(t)-U_p(t)]
\quad\text{for $0\le p\le P$.}
\]
\cref{table: PDE1D cutoff} shows errors in this spatially extrapolated DG
solution over the time interval~$[T/4,T]$, that is,
\begin{equation}\label{eq: cutoff error}
\max_{T/4\le t\le T}\|\bUR(t)-\bs{u}(t)\|_h,
\end{equation}
as well as the corresponding errors in the
reconstruction~$\bUR_*(t)$ and the nodal
values~$(\bUR)^n_-$. Again, the observed convergence rates are
as expected.

\begin{table}
\caption{Maximum errors over the time interval~$[T/4,T]$ for the 1D heat
equation of \cref{sec: PDE 1D} using piecewise-quadratics ($r=3$).}
\label{table: PDE1D cutoff}
\renewcommand{\arraystretch}{1.2}
\begin{center}
\ttfamily
\begin{tabular}{rr|rr|rr|rr}
$N$&$P$&\multicolumn{2}{c|}{\textrm{Error in $\bUR$}}&
\multicolumn{2}{c|}{\textrm{Error in $\bUR_*$}}&
\multicolumn{2}{c}{\textrm{Error in $(\bUR)^n_-$}}\\
\hline
   8 &   500 & 8.46e-05&        & 7.51e-05&        & 5.17e-06\\
  16 &   500 & 1.07e-05&  2.978 & 5.37e-07&  7.129 & 1.15e-07&  5.490\\
  32 &   500 & 1.35e-06&  2.989 & 2.30e-08&  4.544 & 3.99e-09&  4.847\\
  64 &   500 & 1.69e-07&  2.995 & 1.21e-09&  4.244 & 1.47e-10&  4.762\\
 128 &   500 & 2.12e-08&  2.997 & 6.98e-11&  4.121 & 5.80e-12&  4.664\\
\hline
\multicolumn{2}{c|}{\textrm{Theory}}
             &         &     $3$&         &     $4$&         & $5$
\end{tabular}
\end{center}
\end{table}

\begin{table}
\caption{Weighted errors for the 1D heat equation of
\cref{sec: PDE 1D} using piecewise-quadratics ($r=3$) and the indicated
exponent~$\alpha$ in the weight function~$w_\alpha(t)$.  The top set of results
is for the homogeneous equation ($f\equiv0$).  The bottom set is for the
general case (both $u_0$~and $f$ non-zero).}
\label{table: PDE1D weighted}
\renewcommand{\arraystretch}{1.2}
\begin{center}
\ttfamily
\begin{tabular}{rr|rr|rr|rr}
$N$&$P$&\multicolumn{2}{c|}{\textrm{Error in $\bUR$}}&
\multicolumn{2}{c|}{\textrm{Error in $\bUR_*$}}&
\multicolumn{2}{c}{\textrm{Error in $(\bUR)^n_-$}}\\
\hline
&&\multicolumn{2}{c|}{$\alpha=r-\tfrac54$}&
\multicolumn{2}{c|}{$\alpha=r+1-\tfrac54$}&
\multicolumn{2}{c}{$\alpha=2r-1-\tfrac54$}\\
\hline
   8 &   500 & 8.52e-05&        & 7.08e-06&        & 1.77e-06\\
  16 &   500 & 1.15e-05&  2.893 & 4.42e-07&  4.001 & 5.53e-08&  5.001\\
  32 &   500 & 1.49e-06&  2.946 & 2.76e-08&  4.000 & 1.73e-09&  5.000\\
  64 &   500 & 1.90e-07&  2.973 & 1.73e-09&  4.000 & 5.40e-11&  5.000\\
 128 &   500 & 2.39e-08&  2.987 & 1.08e-10&  4.000 & 1.69e-12&  5.000\\
\hline
\multicolumn{2}{c|}{\textrm{Theory}}& &\multicolumn{1}{l|}{$3$}&
&\multicolumn{1}{l|}{$4$}&   &\multicolumn{1}{l}{$5$}
\end{tabular}
\end{center}
\strut
\begin{center}
\ttfamily
\begin{tabular}{rr|rr|rr|rr}
$N$&$P$&\multicolumn{2}{c|}{\textrm{Error in $\UR$}}&
\multicolumn{2}{c|}{\textrm{Error in $\bUR_*$}}&
\multicolumn{2}{c}{\textrm{Error in $(\bUR)^n_-$}}\\
\hline
&&\multicolumn{2}{c|}{$\alpha=r-\tfrac54$}&
\multicolumn{2}{c|}{$\alpha=r+1-\tfrac54$}&
\multicolumn{2}{c}{$\alpha=2r-1-\tfrac54$}\\
\hline
   8 &   500 & 8.46e-05&        & 1.66e-06&        & 4.15e-07\\
  16 &   500 & 1.07e-05&  2.978 & 1.03e-07&  4.007 & 1.70e-08&  4.606\\
  32 &   500 & 1.35e-06&  2.989 & 6.46e-09&  3.999 & 7.89e-10&  4.433\\
  64 &   500 & 1.69e-07&  2.995 & 4.04e-10&  4.000 & 3.84e-11&  4.362\\
 128 &   500 & 2.12e-08&  2.997 & 2.52e-11&  4.000 & 1.93e-12&  4.311\\
\hline
\multicolumn{2}{c|}{\textrm{Theory}}& &\multicolumn{1}{l|}{$3$}&
&\multicolumn{1}{l|}{$4$}&   &\multicolumn{1}{l}{$5$}
\end{tabular}
\end{center}
\end{table}
To investigate the time dependence of the error for~$t$ near zero, we consider
the weighted error in the DG solution
\[
\adjustlimits\max_{1\le n\le N}\sup_{t\in I_n}
    w_\alpha(t)\|\bUR(t)-\bs{u}(t)\|_h
\quad\text{where}\quad w_\alpha(t)=\min(t^\alpha,1),
\]
and likewise incorporate the weight~$w_\alpha(t)$ when measuring
the reconstruction error and the nodal error.  The top part of
\cref{table: PDE1D weighted} shows results for the homogeneous problem, that is,
with the same data as in~\eqref{eq: PDE1D data} except $f(x,t)\equiv0$.  The
$m$th Fourier sine coefficient of~$u_0$ is proportional to~$m^{-3}$, so
$\|A^su_0\|\le C\epsilon^{-1/2}$ for $s=\tfrac54-\epsilon$ and  $\epsilon>0$.
Based on the estimates in \cref{example: weighted error f=0}, we choose the
weight exponents $\alpha=r-\tfrac54$ for the DG error, $r+1-\tfrac54$ for the
reconstruction error, and $2r-1-\tfrac54$ for the nodal error, and observe
excellent agreement in the top set of results in \cref{table: PDE1D weighted}
with the expected convergence rates of order $r$, $r+1$ and $2r-1$,
respectively.

Similar results are found if $u_0(x)\equiv0$ with nonzero~$f$.
Curiously, in the bottom part of~\cref{table: PDE1D weighted}, choosing both
$u_0$~and $f$ as in~\eqref{eq: PDE1D data} (so both nonzero) disturbs the
observed convergence rates for~$(\bUR)^n_-$, although not for
$\bUR$~or $\bUR_*$.

\subsection{A parabolic PDE in 2D}\label{sec: PDE 2D}

\begin{table}
\caption{Maximum errors over the time interval~$[T/4,T]$ for the spatially
discrete, 2D heat equation of \cref{sec: PDE 2D} using piecewise-quadratics
($r=3$).}
\label{table: PDE2D cutoff}
\renewcommand{\arraystretch}{1.2}
\begin{center}
\ttfamily
\begin{tabular}{rrr|rr|rr|rr}
$N$&$P_x$&$P_y$&\multicolumn{2}{c|}{\textrm{Error in $\bs{U}_h$}}&
\multicolumn{2}{c|}{\textrm{Error in $(\bs{U}_h)_*$}}&
\multicolumn{2}{c}{\textrm{Error in $(\bs{U}_h)^n_-$}}\\
\hline
   8 &    50 &    50 & 5.32e-04&        & 4.70e-04&        & 2.60e-05&       \\
  16 &    50 &    50 & 4.60e-05&  3.533 & 1.48e-06&  8.316 & 4.40e-07&  5.888\\
  32 &    50 &    50 & 5.15e-06&  3.160 & 6.80e-08&  4.440 & 1.43e-08&  4.940\\
  64 &    50 &    50 & 6.10e-07&  3.078 & 4.16e-09&  4.029 & 4.65e-10&  4.944\\
 128 &    50 &    50 & 7.42e-08&  3.038 & 2.58e-10&  4.010 & 1.49e-11&  4.967\\
\hline
\multicolumn{3}{c|}{\textrm{Theory}}
             &         & \multicolumn{1}{l|}{$3$}
             &         & \multicolumn{1}{l|}{$4$}
             &         & \multicolumn{1}{l}{$5$}
\end{tabular}
\end{center}
\end{table}

Now consider the 2D heat equation,
\begin{equation}
u_t-\kappa\nabla^2 u=f(x,y,t)\quad
    \text{for $0<t\le T$ and $(x,y)\in\Omega=(0,L_x)\times(0,L_y)$,}
\end{equation}
subject to the boundary conditions $u(x,y,t)=0$ for $(x,y)\in\partial\Omega$,
and to the initial condition $u(x,y,0)=u_0(x,y)$ for $(x,y)\in\Omega$.
We introduce a spatial finite difference grid
\[
(x_p,y_q)=(p\,h_x,q\,h_y)\quad
    \text{for $0\le p\le P_x$ and $0\le q\le P_y$,}
\]
with $h_x=L_x/P_x$~and $h_y=L_y/P_y$.  The semidiscrete finite difference
solution~$u_{pq}(t)\approx u(x_p,y_q,t)$ is then constructed using the standard
5-point approximation to the Laplacian, so that
\begin{equation}\label{eq: finite diff}
u_{pq}'-\kappa\biggl(\frac{u_{p+1,q}-2u_{pq}+u_{p-1,q}}{h_x^2}
    +\frac{u_{p,q+1}-2u_{pq}+u_{p,q-1}}{h_y^2}\biggr)=f_{pq}
\end{equation}
for $0\le t\le T$ and $(x_p,y_q)\in\Omega$, where $f_{pq}(t)=f(x_p,y_q,t)$,
together with the boundary condition $u_{pq}(t)=0$
for~$(x_p,y_q)\in\partial\Omega$, and the initial
condition~$u_{pq}(0)=u_0(x_p,y_q)$ for~$(x_p,y_q)\in\Omega$.
For $(x_p,y_q)\in\Omega$, we use column-major ordering to arrange the unknowns
$u_{pq}(t)$, the source terms $f_{pq}(t)$ and initial data $u_{0pq}$ into
vectors~$\bs{u}_h(t)$, $\bs{f}(t)$~and $\bs{u}_0\in\mathbb{R}^M$
for~$M=(P_x-1)(P_y-1)$.   There is then a sparse matrix~$\bs{A}$ such that the
system of ODEs~\eqref{eq: finite diff} leads to the initial-value problem
\begin{equation}\label{eq: semidiscrete}
\bs{u}_h'(t)+\bs{A}\bs{u}_h=\bs{f}(t)\quad\text{for $0\le t\le T$, with
$\bs{u}_h(0)=\bs{u}_0$.}
\end{equation}

For our test problem, we take $L_x=L_y=2$~and $P_x=P_y=50$ with
\begin{equation}\label{eq: PDE2D data}
T=2,\quad\kappa=2/\pi^2,\quad u_0(x,y)=x(2-x)y(2-y),\quad
f(x,y,t)=(1+t)e^{-t},
\end{equation}
where the choice of~$\kappa$ ensures that the smallest Dirichlet eigenvalue of
$-\kappa\nabla^2$ on~$\Omega$ equals~$1$.  \cref{table: PDE2D cutoff}
compares the piecewise-quadratic ($r=3$) DG solution~$\bs{U}_h(t)$ of the
semidiscrete problem~\eqref{eq: semidiscrete} with~$\bs{u}_h(t)$, evaluating the
latter using numerical inversion of the Laplace transform as before except that
now, instead of~$\hat u(z)$, we work with the spatially discrete
approximation~$\hat{\bs{u}}_h(z)$ obtained by solving the (complex) linear
system $(z\bs{I}+\bs{A})\hat{\bs{u}}_h(z)=u_0+\hat{\bs{f}}(z)$.  As with the 1D
results in \cref{table: PDE1D cutoff}, we compute the maximum error over the
time interval~$[T/4,T]$, and observe the expected rates of convergence, keeping
in mind that by treating  $\bs{u}_h(t)$ as our reference solution we are
ignoring the error from the spatial discretization.
\section{Declarations}
\subsection{Ethical Approval and Consent to participate}
Not applicable.
\subsection{Consent for publication}
Not applicable.
\subsection{Human and Animal Ethics}
Not applicable.
\subsection{Availability of supporting data}
The paper does not make use of any data sets.  The software used to generate
the numerical results is available on github~\cite{McLean2022}.
\subsection{Competing interests}
The authors have no competing interests.
\subsection{Funding}
This work was not funded as part of any research grant.
\subsection{Authors' contributions}
William McLean wrote an initial outline of the paper, that subsequently
underwent multiple revisions arising from correspondence with Kassem Mustapha.
William McLean carried out the numerical computations reported in the paper.
\subsection{Acknowledgments}
None.
\bibliography{DG_superconv}
\end{document}